\documentclass[amsthm,secthm,seceqn]{elsart1p}
\usepackage{amsmath}
\usepackage{amssymb,amsmath,mathrsfs}
\usepackage{cases,cite}
\usepackage{color}
\usepackage{graphicx}
\usepackage{subfigure}
\newtheoremstyle{mythm}{4pt}{4pt}{\itshape}{}{\bfseries}{.}{.5em}{}
\theoremstyle{mythm}
\newtheorem{lemma}{Lemma}[section]
\newtheorem{theorem}{Theorem}[section] 

\newtheoremstyle{myrek}{4pt}{4pt}{}{}{\itshape}{.}{0.5em}{}
\theoremstyle{definition}
\newtheorem{definition}{Definition}[section]  
\theoremstyle{myrek}
\newtheorem{remark}{\bf Remark}[section]
\allowdisplaybreaks
\AtBeginDocument{\rule{0ex}{1ex}\vskip-3cm\mag=1200}
\begin{document}
\begin{frontmatter}
\title{{\textbf{Distributivity between extended nullnorms and uninorms on fuzzy truth values}}
\thanksref{thk}}

\thanks[thk]{This work is supported by National Natural Science Foundation of China (No. 11171242).}

\author{Zhi-qiang Liu},
\author{Xue-ping Wang\corauthref{cor}}\qquad
\corauth[cor]{Corresponding author. Fax: +86 28 84761393.}
\ead{xpwang1@hotmail.com}
\address{School of Mathematical Sciences, Sichuan Normal University, Chengdu, Sichuan 610066, P.R. China}
\begin{abstract}
This paper mainly investigates the distributive laws between extended nullnorms and uninorms on fuzzy truth values under the condition that the nullnorm is conditionally distributive over the uninorm. It presents the distributive laws between the extended nullnorm and t-conorm, and the left and right distributive laws between the extended generalization nullnorm and uninorm, where a generalization nullnorm is an operator from the class of aggregation operators with absorbing element that generalizes a nullnorm.
\end{abstract}
\end{frontmatter}

\section{Introduction}

The concept of a type-2 fuzzy set was introduced by Zadeh in 1975 \cite{Zadeh1975} as an extension of type-1 fuzzy sets, and it has been heavily
investigated both as a mathematical object and for use in applications \cite{Walker2005,Zadeh1975}. The algebra of truth values for fuzzy sets of type-2 consists of all mappings from the unit interval into itself and their operations which are convolutions of operations on the unit interval \cite{Walker2005}.
The algebra theory was studied extensively by Harding, C. and E. Walker \cite{Harding2016}, and C. and E. Walker \cite{Walker2005,WALKER2006,Walker2009}.
Theory of aggregation of real numbers play an important role in many different theoretical and practical fields, e.g., decision making theory, fuzzy set theory, integration theory, ect. Aggregation operators for real numbers are extended to ones for type-2 fuzzy sets. For example,
Gera and Dombi \cite{Gera20081} proposed computationally simple, pointwise formulas for extended t-norms and t-conorms on fuzzy truth values;
Tak\'{a}\u{c} \cite{Takac2014} investigated extended aggregation operations on the algebra of convex normal fuzzy truth values with their left and right parts;
Torres-Blanc, Cubillo, and Hern\'{a}ndez \cite{Torres2017} applied the Zadeh's extension principle to extend the aggregation operations of type-1 to the case of tyep-2 fuzzy sets. In particular, the distributive laws between those convolution operations on fuzzy truth values become an interesting and natural research area, so that they are discussed in many articles.
For instance, Harding, C. and E. Walker \cite{Harding2016} and C. and E. Walker \cite{Walker2005,Walker2009} discussed the distributive laws between extended minimums and maximums, and extended maximums and minimums, respectively, the distributive laws between extended t-norms and maximums, and the distributive laws between extended t-conorms and minimums.
Hu and Kwong \cite{HuBQ2014} also presented the distributive laws between extended t-norms and maximums, and the distributive laws between extended t-conorms and minimums. Xie \cite{Xie2018} extended type-1 proper nullnorms and proper uninorms to fuzzy truth values and studied the distributive laws between the extended uninorms and minimums, and the distributive laws between the extended uninorms and maximums.
Recently, Liu and Wang \cite{Liu2019} discussed distributivity between extended t-norms and t-conorms on fuzzy truth values under the condition that the t-norm is conditionally distributive over the t-conorm or the t-conorm is conditionally distributive over the t-norm. It is well known that uninorms \cite{Yager1996} and nullnorms \cite{Calvo2001} are aggregation operations with neutral elements and absorbing elements on $[0,1]$, respectively. They are generalizations of t-norms and t-conorms as well. However, the distributive laws between the extended nullnorms and uninorms on fuzzy truth values are not discussed till now, so that this paper will investigate these problems based on the results of conditionally distributivity of nullnorms over the uninorms in \cite{Dragan2013,Dragan2015,LiG2015}.

This paper is organized as follows. In Section 2 we recall some necessary definitions and previous results.
In Section 3 we investigate the distributive laws between extended nullnorms and uninorms on fuzzy truth values under the condition that the nullnorm is conditionally distributive over the uninorm. In Section 4 we study distributivity of extended continuous operators with absorbing element and extended uninorms.
A conclusion is given in Section 5.

\section{Previous Results}

In this section, we recall some basic concepts and terminologies used throughout the paper.

\begin{definition}[\cite{Klement2000}]
A t-norm (resp. t-conorm) is a binary operation $T:[0,1]^{2}\rightarrow[0,1]$ (resp. $S:[0,1]^{2}\rightarrow[0,1]$) that is commutative, associative, non-decreasing in each variable, and has a neutral element $1$ (resp. $0$).
\end{definition}

\begin{definition}[\cite{Klement2000}]
\quad

(i) A t-norm $T$ is said to be strict, if $T$ is continuous and strictly monotone.

(ii) A t-norm $T$ is said to be nilpotent, if $T$ is continuous and if each $x\in(0,1)$ is a nilpotent element of $T$.
\end{definition}

The basic continuous t-norms are minimum, $T_{M}(x,y)=\min(x,y)$, the product, $T_{P}(x,y)=xy$, and the {\L}ukasiewicz t-norm, $T_{L}(x,y)=\max(x+y-1,0)$. Dually, the basic continuous t-conorms are maximum, $S_{M}(x,y)=\max(x,y)$, the probabilistic sum, $S_{P}(x,y)=x+y-xy$, and the {\L}ukasiewicz t-conorm, $S_{L}(x,y)=\min(x+y,1)$.

\begin{definition}[\cite{Klement2000}]
A binary function $U : [0,1]^{2}\rightarrow[0,1]$ is called a uninorm  if it is commutative, associative, non-decreasing in each place and there exists
some element $e\in[0,1]$ such that $U(x,e)=x$ for all $x\in[0,1]$ where $e$ is called a neutral element of $U$.
\end{definition}

One can see that a uninorm $U$ is a t-norm if $e=1$, and a t-conorm if $e=0$. A uninorm $U$ is called proper if its neutral element $e\in(0,1)$. It is clear that $U(0,1)\in\{0,1\}$ (see \cite{Fodor1997}). $U$ is said to be conjunctive if $U(1,0)=0$, and be disjunctive if $U(1,0)=1$.

With any uninorm $U$ with neutral element $e\in(0,1)$, we can associate two binary operations $T_{U}$ and $S_{U}: [0,1]^{2}\rightarrow [0,1]$ defined by
\begin{eqnarray*}
T_{U}(x,y)=\frac{U(ex,ey)}{e}    \mbox{ \ \ \ and \ \ \ }
S_{U}(x,y)=\frac{U((e+(1-e)x,e+(1-e)y))-e}{1-e},
\end{eqnarray*}
respectively. It is easy to see that $T_{U}$ is a t-norm and that $S_{U}$ is a t-conorm where $T_{U}$ is called an underlying t-norm, and $S_{U}$ is called an underlying t-conorm.
Let us denote the remaining part of the unit square by $E$, i.e., $E=[0,1]^{2}\setminus([0,e]^{2}\cup[e,1]^{2})$. On the set $E$, any uninorm $U$ is bounded by the minimum and maximum of its arguments, i.e., for any $(x,y)\in E$, $$\min(x,y)\leqslant U(x,y)\leqslant \max(x,y).$$

The most studied classes of uninorms are:

$\bullet$ Idempotent uninorms in $\mathcal{U}_{id}$\cite{Baets,LiG2015}, those that satisfy $U(x,x)=x$ for all $x\in[0,1]$.

$\bullet$ Uninorms in $\mathcal{U}_{\min}$ (resp. $\mathcal{U}_{\max}$)\cite{Fodor1997,LiG2015}, those given by minimum (resp. maximum) in $E$.

$\bullet$ Uninorms in $\mathcal{CU}$\cite{HuSK2001,LiG2015}, those that are continuous in the open square $(0,1)^{2}$.

$\bullet$ Uninorms in $\mathcal{WCU}$\cite{Calvo20015,Saminger2007,LiG201501,LiG2015}, those that are with a continuous underlying t-norm and t-conorm.

\begin{definition}[\cite{Calvo2001,Klement2000}]\label{defi2.4}
A nullnorm is a binary operation $F : [0,1]^{2}\rightarrow[0,1]$, which is commutative, associative, non-decreasing in each variable and there exists an
element $k\in[0,1]$ such that $F(0,x)=x$ for all $x\in [0, k]$ and $F(1,x)=x$ for all $x\in [k,1]$.
\end{definition}

Clearly, a nullnorm $F$ is a t-norm if $k=0$, and a t-conorm if $k=1$. If $k\in(0,1)$, then a nullnorm $F$ is called proper. It is immediately clear that every nullnorm $F$ satisfies $F(k,x)=k$ for all $x\in[0,1]$, i.e., $k$ is an absorbing element of $F$.

\begin{definition}[\cite{Walker2005}]
Fuzzy truth values are mapping of $[0,1]$ into itself. The set of fuzzy truth values is denoted by $\mathcal{F}=\{f \mid f: [0,1]\rightarrow [0,1]\}$.
\end{definition}

According to Zadeh's extension principle, a two-place function $G: [0,1]^{2}\rightarrow[0,1]$ can be extended to $\odot_{G}:\mathcal{F}^{2}\rightarrow\mathcal{F}$ by the convolution of $G$ with respect to $\wedge$ and $\vee$. Let $f,g\in\mathcal{F}$. Then $(f\odot_{G} g)(z)=\bigvee\limits_{z=G(x,y)}(f(x)\wedge g(y))$. Here, $\odot_{G}$ is called the extended $G$.

If $G$ is a nulllnorm $F$ or a uninorm $U$, then we have its extended nullnorm or uninorm defined by
\begin{eqnarray}\label{eqn41}
(f\odot_{F} g)(z)=\bigvee_{F(x,y)=z}(f(x)\wedge g(y))
\end{eqnarray}
and
\begin{eqnarray}\label{eqn42}
(f\odot_{U} g)(z)=\bigvee_{U(x,y)=z}(f(x)\wedge g(y)),
\end{eqnarray}
 respectively (see \cite{Xie2018}). In particular, if $G$ is the t-norm $T_{M}=\min$ or t-conorm $S_{M}=\max$, then we use $\sqcap$ and $\sqcup$ instead of $\odot_{F}$ and $\odot_{U}$, respectively (see \cite{Walker2005}), i.e.,
\begin{eqnarray}\label{eq413}
(f\sqcap g)(z)=\bigvee\limits_{x\wedge y=z}(f(x)\wedge g(y)),
\end{eqnarray}
\begin{eqnarray}\label{eq414}
(f\sqcup g)(z)=\bigvee\limits_{x\vee y=z}(f(x)\wedge g(y)).
\end{eqnarray}

\begin{definition}[\cite{Walker2005}]\label{conves1}
An element $f\in \mathcal{F}$ is said to be convex if for all $x,y,z\in[0,1]$ for which $x\leqslant y\leqslant z$, we have $f(y)\geqslant f(x)\wedge f(z)$.
\end{definition}

\section{Distributive laws between the extended nullnorms and uninorms}

In this section, the distributive laws between the extended nullnorms and uninorms on fuzzy truth values are discussed.

\begin{theorem}\label{theor41}
Let $F:[0,1]^{2}\rightarrow[0,1]$ be a continuous non-decreasing operator. If $f\in \mathcal{F}$ is convex, then the following statements hold for all $g,h\in\mathcal{F}$.
\begin{enumerate}
\item [(i)] $f\odot_{F}(g\sqcap h)=(f\odot_{F}g)\sqcap(f\odot_{F}h)$;
\item [(ii)] $f\odot_{F}(g\sqcup h)=(f\odot_{F}g)\sqcup(f\odot_{F}h)$.
\end{enumerate}
\end{theorem}
\begin{proof}
We only provide the proof of statement (i), the statement of (ii) being analogous.

According to formulas (\ref{eqn41}) and (\ref{eq413}), for all $z\in[0,1]$, we have
\begin{eqnarray*}
\left(f\odot_{F}(g\sqcap h)\right)(z)=\bigvee\limits_{F(y,u\wedge v)=z}f(y)\wedge g(u)\wedge h(v)
\end{eqnarray*}
and
\begin{eqnarray*}
\left((f\odot_{F}g)\sqcap(f\odot_{F}h)\right)(z)&=&\bigvee\limits_{F(p,q)\wedge F(s,t)=z}f(p)\wedge g(q)\wedge f(s)\wedge h(t).
\end{eqnarray*}
On one hand, due to $F(y,u\wedge v)=F(y,u)\wedge F(y,v)$, it holds that
\begin{eqnarray}\label{eqna40}
\left(f\odot_{F}(g\sqcap h)\right)(z)\leqslant\left((f\odot_{F}g)\sqcap(f\odot_{F}h)\right)(z).
\end{eqnarray}
On the other hand, suppose that $z=F(p,q)\wedge F(s,t)$. It is easy to see if there exists $y\in[0,1]$ such that both of the following hold then the reverse inequality of (\ref{eqna40}) is holds.
\begin{eqnarray}\label{eqna41}
z=F(y,q\wedge t)
\end{eqnarray}
and
\begin{eqnarray}\label{eqna42}
f(y)\wedge g(q)\wedge h(t)\geqslant f(p)\wedge g(q)\wedge f(s)\wedge h(t).
\end{eqnarray}

Next, we shall prove formulas (\ref{eqna41}) and (\ref{eqna42}). From $z=F(p,q)\wedge F(s,t)$, we distinguish three cases.

(i) If $F(p,q)=F(s,t)=z$, then let $y=p\wedge s$. Thus $F(y,q\wedge t)=F(y,q)\wedge F(y,t)=z$, and $f(y)=f(p)$ or $f(y)=f(s)$. Therefore, the formulas (\ref{eqna41}) and (\ref{eqna42}) are hold.

(ii) If $F(p,q)>z$ and $F(s,t)=z$, then we have the following two subcases.

\begin{enumerate}
  \item []
  \begin{enumerate}
  \item [{\rm $\bullet$}] If $F(s,q)\geqslant z$, then $F(s,q)\wedge F(s,t)=F(s,t)=z$, and put $y=s$. Then both (\ref{eqna41}) and (\ref{eqna42}) hold.
  \item [{\rm $\bullet$}] If $F(s,q)<z$, then $F(s,q)<F(s,t)$ means $q<t$. Moreover, $F(p,q)>z>F(s,q)$ implies $s<p$. Because $F$ is continuous, there exists a $y$ with $s<y<p$ such that $F(y,q)=z$, and $f(y)\geqslant f(p)\wedge f(s)$ since $f$ is convex. Hence both (\ref{eqna41}) and (\ref{eqna42}) hold.
  \end{enumerate}
\end{enumerate}

(iii) If $F(p,q)=z$ and $F(s,t)>z$, then similar to (ii), we can get that (\ref{eqna41}) and (\ref{eqna42}).

With Cases (i), (ii) and (iii), we always know that both (\ref{eqna41}) and (\ref{eqna42}) hold. Therefore,
$$\left(f\odot_{F}(g\sqcap h)\right)(z)\geqslant\left((f\odot_{F}g)\sqcap(f\odot_{F}h)\right)(z).$$

This completes the proof.
\end{proof}

Note that Theorem \ref{theor41} generalizes the sufficiency of Proposition 3.9 in \cite{HuBQ2014} (see also the conclusions 1 and 4 of Theorem 5.5.3 in \cite{Harding2016}).

\begin{lemma}[\cite{Martín2003}]\label{le3.1}
Consider $e\in(0,1)$. The following statements are equivalent:
\begin{enumerate}
\item [(i)] $U$ is an idempotent uninorm with neutral element $e$.
\item [(ii)] There exists a non-increasing function $g:[0,1]\rightarrow[0,1]$, symmetric with respect to the main diagonal, with $g(e)=e$, such that, for all $(x,y)\in E$
    \begin{eqnarray}
U(x,y)=\left\{\begin{array}{ll}
\min(x,y),    & {\mbox{\scriptsize\normalsize if }y<g(x) \mbox{ or }y=g(x) \mbox{ and }x<g(g(x)),}\\
\max(x,y),    & {\mbox{\scriptsize\normalsize if }y>g(x) \mbox{ or }y=g(x) \mbox{ and }x>g(g(x)),}\\
x \mbox{ or } y,  & {\mbox{\scriptsize\normalsize if }y=g(x)   \mbox{ and }x=g(g(x))}\\
\end{array}
\right.
\end{eqnarray}
being commutative on the set of points $(x,y)$ such that $y=g(x)$ with $x=g(g(x))$.
\end{enumerate}
\end{lemma}

\begin{remark}\label{remark4}
The first uninorms, which were constructed by Yager and Rybalov\cite{Yager1996}, are idempotent uninorms from classes $\mathcal{U}_{\min}$ and $\mathcal{U}_{\max}$ of the following form:
\begin{eqnarray*}
\underline{U}(x,y)=\left\{\begin{array}{ll}
\max(x,y) , & {\mbox{\scriptsize\normalsize if }(x,y)\in[e,1]^{2},}\\
\min(x,y),    &  {\mbox{\scriptsize\normalsize otherwise}}
\end{array}
\right.
\end{eqnarray*}
and
\begin{eqnarray*}
\overline{U}(x,y)=\left\{\begin{array}{ll}
\min(x,y) , & {\mbox{\scriptsize\normalsize if }(x,y)\in[0,e]^{2},}\\
\max(x,y),    &  {\mbox{\scriptsize\normalsize otherwise. }}
\end{array}
\right.
\end{eqnarray*}
\end{remark}
These uninorms are the only explicit examples of idempotent uninorms.

From Theorem \ref{theor41} and Lemma \ref{le3.1}, we immediately have the following result.
\begin{theorem}
Let $F$ be a continuous nulllnorm and $U$ an idempotent uninorm. If $f\in \mathcal{F}$ is convex, then the following holds for all $g,h\in\mathcal{F}$.
$$f\odot_{F}(g\odot_{U} h)=(f\odot_{F}g)\odot_{U}(f\odot_{F}h).$$
\end{theorem}

\begin{definition}
Let $F:[0,1]^{2}\rightarrow[0,1]$ be a continuous non-decreasing operator and $U$ be a uninorm.
  \begin{enumerate}
  \item [{\rm $\bullet$}] $F$ is conditionally distributive over $U$ from the left (CDl) if $F(x,U(y,z))=U(F(x,y),F(x,z))$ for all $x,y,z\in[0,1]$ whenever $U(y,z)<1$.

  \item [{\rm $\bullet$}] $F$ is conditionally distributive over $U$ from the right (CDr) if $F(U(x,y),z)=U(F(x,z),F(y,z))$ for all $x,y,z\in[0,1]$ whenever $U(x,y)<1$.
\end{enumerate}
\end{definition}

Of course, for a commutative operator $F$ (CDl) and (CDr) coincides and are denoted by (CD).

In the sequel, we study the distributive laws between the extended nullnorm and uninorm on fuzzy truth values under the condition that the nullnorm is conditionally distributive over the uninorm. First, we need the lemma as follows.

\begin{lemma}[\cite{LiG2015}]\label{theor44}
A continuous nullnorm $F$ with an absorbing element $k\in(0,1)$ and a disjunctive uninorm $U\in\mathcal{WCU}$ with neutral element $e\in(0,1)$ satisfy (CD) if and only if one of the following cases is fulfilled:

\begin{enumerate}
  \item []
  \begin{enumerate}
  \item [{\rm (i)}]
$e<k$ and $F,U$ are given as in \cite{Mas2002} (Proposition 4.2), i.e.,
\begin{eqnarray}\label{eqn45}
U(x,y)=\left\{\begin{array}{ll}
\min(x,y) , & {\mbox{\scriptsize\normalsize if }(x,y)\in[0,e]^{2},}\\
\max(x,y),    &  {\mbox{\scriptsize\normalsize otherwise }}
\end{array}
\right.
\end{eqnarray}
and
\begin{eqnarray}\label{eqn46}
F(x,y)=\left\{\begin{array}{ll}
eS_{1}(\frac{x}{e},\frac{y}{e}) , & {\mbox{\scriptsize\normalsize if }(x,y)\in[0,e]^{2},}\\
e+(k-e)S_{2}\left(\frac{x-e}{k-e},\frac{y-e}{k-e}\right),    & {\mbox{\scriptsize\normalsize if }(x,y)\in[e,k]^{2},}\\
k+(1-k)T\left(\frac{x-k}{1-k},\frac{y-k}{1-k}\right),    & {\mbox{\scriptsize\normalsize if }(x,y)\in[k,1]^{2},}\\
\max(x,y),    & {\mbox{\scriptsize\normalsize if }\min(x,y)\leqslant e\leqslant \max(x,y)\leqslant k,}\\
k,    &  {\mbox{\scriptsize\normalsize otherwise, }}
\end{array}
\right.
\end{eqnarray}
 where $S_{1}$ and $S_{2}$ are continuous t-conorms and $T$ is a continuous t-norm.
\end{enumerate}
\begin{enumerate}
  \item [{\rm (ii)}]
  $e<k$ and $F,U$ are given as in \cite{Dragan2013} (Theorem 16), i.e., there is $a\in[k,1)$ such that $F$ and $U$ are given by
  \begin{eqnarray}\label{eqn47}
U(x,y)=\left\{\begin{array}{ll}
\min(x,y) , & {\mbox{\scriptsize\normalsize if }(x,y)\in[0,e]^{2},}\\
a+(1-a)S\left(\frac{x-a}{1-a},\frac{y-a}{1-a}\right),    & {\mbox{\scriptsize\normalsize if }(x,y)\in[a,1]^{2},}\\
\max(x,y),    &  {\mbox{\scriptsize\normalsize otherwise }}
\end{array}
\right.
\end{eqnarray}
and
\begin{eqnarray}\label{eqn48}
F(x,y)=\left\{\begin{array}{ll}
eS_{1}(\frac{x}{e},\frac{y}{e}) , & {\mbox{\scriptsize\normalsize if }(x,y)\in[0,e]^{2},}\\
e+(k-e)S_{2}\left(\frac{x-e}{k-e},\frac{y-e}{k-e}\right),    & {\mbox{\scriptsize\normalsize if }(x,y)\in[e,k]^{2},}\\
k+(a-k)T_{1}\left(\frac{x-k}{a-k},\frac{y-k}{a-k}\right),    & {\mbox{\scriptsize\normalsize if }(x,y)\in[k,a]^{2},}\\
a+(1-a)T\left(\frac{x-a}{1-a},\frac{y-a}{1-a}\right),    & {\mbox{\scriptsize\normalsize if }(x,y)\in[a,1]^{2},}\\
\max(x,y),    & {\mbox{\scriptsize\normalsize if } \min(x,y)\leqslant e\leqslant \max(x,y)\leqslant k,}\\
\min(x,y),    & {\mbox{\scriptsize\normalsize if } k\leqslant\min(x,y)\leqslant a\leqslant\max(x,y),}\\
k,    &  {\mbox{\scriptsize\normalsize otherwise, }}
\end{array}
\right.
\end{eqnarray}
 where $S_{1}$ and $S_{2}$ are continuous t-conorms, $T_{1}$ is a continuous t-norm and $S$ is a nilpotent t-conorm such that the additive generator $s$ of $S$ satisfying $s(1)=1$ is also a multiplicative generator of the strict t-norm $T$.
  \end{enumerate}

  \begin{enumerate}
  \item [{\rm (iii)}]
  $e>k$ and
  \begin{eqnarray}\label{eqn49}
U(x,y)=\left\{\begin{array}{ll}
\max(x,y) , & {\mbox{\scriptsize\normalsize if }(x,y)\in[e,1]^{2},}\\
1,    & {\mbox{\scriptsize\normalsize if }x=1 \mbox{ or } y=1,}\\
\min(x,y),    &  {\mbox{\scriptsize\normalsize otherwise }}
\end{array}
\right.
\end{eqnarray}
and
\begin{eqnarray}\label{eqn410}
F(x,y)=\left\{\begin{array}{ll}
kS_{1}(\frac{x}{k},\frac{y}{k}) , & {\mbox{\scriptsize\normalsize if }(x,y)\in[0,k]^{2},}\\
k+(e-k)T_{1}\left(\frac{x-k}{e-k},\frac{y-k}{e-k}\right),    & {\mbox{\scriptsize\normalsize if }(x,y)\in[k,e]^{2},}\\
e+(1-e)T_{2}\left(\frac{x-e}{1-e},\frac{y-e}{1-e}\right),    & {\mbox{\scriptsize\normalsize if }(x,y)\in[e,1]^{2},}\\
k,    & {\mbox{\scriptsize\normalsize if }\min(x,y)\leqslant k\leqslant \max(x,y),}\\
\min(x,y),    &  {\mbox{\scriptsize\normalsize otherwise, }}
\end{array}
\right.
\end{eqnarray}
 where $S_{1}$ is continuous t-conorm and $T_{1}$ and $T_{2}$ are continuous t-norms.
  \end{enumerate}
  \end{enumerate}
\end{lemma}

Then we have the following theorem.
\begin{theorem}\label{5theo3.4}
Let $F$ be a continuous nullnorm with an absorbing element $k\in(0,1)$ and $U\in\mathcal{WCU}$ a disjunctive uninorm with neutral element $e\in(0,1)$ satisfying (CD). If $f\in \mathcal{F}$ is convex, then the following holds for all $g,h\in\mathcal{F}$.
$$f\odot_{F}(g\odot_{U} h)=(f\odot_{F}g)\odot_{U}(f\odot_{F}h).$$
\end{theorem}

\begin{proof}
First, from formulas (\ref{eqn41}) and (\ref{eqn42}), we have that
\begin{eqnarray}\label{eqn411}
\left(f\odot_{F}(g\odot_{U} h)\right)(z)=\bigvee\limits_{F(y,U(u,v))=z}f(y)\wedge g(u)\wedge h(v)
\end{eqnarray}
and
\begin{eqnarray}\label{eqn412}
\left((f\odot_{F}g)\odot_{U}(f\odot_{F}h)\right)(z)&=&\bigvee\limits_{U(F(p,q),F(s,t))=z}f(p)\wedge g(q)\wedge f(s)\wedge h(t).
\end{eqnarray}

Now, suppose that $z\in [0,1)$. Then from $z=F(y,U(u,v))\in [0,1)$, we have $U(u,v)<1$ for any $u,v\in [0,1]$. Thus $F(y,U(u,v))=U(F(y,u),F(y,v))\in [0,1)$ since $F$ and $U$ satisfy (CD).
Next, we divide our proof into three cases as follows from Lemma \ref{theor44}.

(i) $e<k$ and $U,F$ are given as Eqs. (\ref{eqn45}) and (\ref{eqn46}), respectively. From Theorem \ref{theor41} (i) and (ii), it is obvious that $\left(f\odot_{F}(g\odot_{U} h)\right)(z)=\left((f\odot_{F}g)\odot_{U}(f\odot_{F}h)\right)(z)$ for $z\in[0,1)$.

(ii) $e<k$ and $U,F$ are given as Eqs. (\ref{eqn47}) and (\ref{eqn48}), respectively. In the following, we shall prove $\left(f\odot_{F}(g\odot_{U} h)\right)(z)=\left((f\odot_{F}g)\odot_{U}(f\odot_{F}h)\right)(z)$ for $z\in[0,1)$. This will be done by checking the subsequent four cases.

Case (a). If $z=F(y,U(u,v))\in[0,e]$, then $U(u,v)=\min(u,v)$. Therefore, $\left(f\odot_{F}(g\odot_{U} h)\right)(z)=\left((f\odot_{F}g)\odot_{U}(f\odot_{F}h)\right)(z)$ from Theorem \ref{theor41} (i).

Case (b). If $z=F(y,U(u,v))\in[e,k]$, then we distinguish three subcases.
\begin{enumerate}
  \item []
  \begin{enumerate}
  \item [{\rm $\bullet$}] If $y\in[e,k]$ and $U(u,v)\in[e,k]$, then $U(u,v)=\max(u,v)$. Therefore, $\left(f\odot_{F}(g\odot_{U} h)\right)(z)=\left((f\odot_{F}g)\odot_{U}(f\odot_{F}h)\right)(z)$ from Theorem \ref{theor41} (ii).
  \item [{\rm $\bullet$}] If $y\in[e,k]$ and $U(u,v)\in[0,e]$, then $\min(y,U(u,v))\leqslant e\leqslant \max(y,U(u,v))\leqslant k$, which implies $F(y,U(u,v))=\max(y,U(u,v))$, and $U(u,v)=\min(u,v)$. Therefore, $\left(f\odot_{F}(g\odot_{U} h)\right)(z)=\left((f\odot_{F}g)\odot_{U}(f\odot_{F}h)\right)(z)$ from Theorem \ref{theor41} (i).
  \item [{\rm $\bullet$}] If $y\in[0,e]$ and $U(u,v)\in[e,k]$, then $\min(y,U(u,v))\leqslant e\leqslant \max(y,U(u,v))\leqslant k$, which means that $F(y,U(u,v))=\max(y,U(u,v))$, and $U(u,v)=\max(u,v)$. Therefore, $\left(f\odot_{F}(g\odot_{U} h)\right)(z)=\left((f\odot_{F}g)\odot_{U}(f\odot_{F}h)\right)(z)$ from Theorem \ref{theor41} (ii).
  \end{enumerate}
\end{enumerate}

Case (c). If $z=F(y,U(u,v))\in[k,a]$, then we distinguish three subcases.
\begin{enumerate}
  \item []
  \begin{enumerate}
  \item [{\rm $\bullet$}] If $y\in[k,a]$ and $U(u,v)\in[k,a]$, then $U(u,v)=\max(u,v)$. Therefore, $\left(f\odot_{F}(g\odot_{U} h)\right)(z)=\left((f\odot_{F}g)\odot_{U}(f\odot_{F}h)\right)(z)$ from Theorem \ref{theor41} (ii).
  \item [{\rm $\bullet$}] If $y\in[k,a]$ and $U(u,v)\in(a,1)$, then $k\leqslant\min(y,U(u,v))\leqslant a\leqslant\max(y,U(u,v))$, which implies $F(y,U(u,v))=\min(y,U(u,v))$. Since $U(u,v)\in(a,1)$, we have either $F=\min$ and $U=\max$ or $F=\min$ and $U(u,v)=a+(1-a)S\left(\frac{u-a}{1-a},\frac{v-a}{1-a}\right)$ with $u,v\in(a,1)$.
      \begin{enumerate}
      \item [{\rm ($\ast$)}] If $F=\min$ and $U=\max$, then $$\left(f\odot_{F}(g\odot_{U} h)\right)(z)=\left((f\odot_{F}g)\odot_{U}(f\odot_{F}h)\right)(z)$$ from Theorem \ref{theor41} (ii).
      \item [{\rm ($\ast\ast$)}] If $F=\min$, and $U(u,v)=a+(1-a)S\left(\frac{u-a}{1-a},\frac{v-a}{1-a}\right)$ with $u,v\in(a,1)$, then $z=y=F(y,U(u,v))=U(F(y,u),F(y,v))=U(y,y)$, i.e., $y=U(y,y)$ since $y< U(u,v)$, which means that $y$ is an idempotent element of $U$, contrary to $U(u,v)= a+(1-a)S\left(\frac{u-a}{1-a},\frac{v-a}{1-a}\right)$ with $u,v\in(a,1)$. Therefore, this subcase is not possible.
     \end{enumerate}
  \item [{\rm $\bullet$}] If $y\in[a,1]$ and $U(u,v)\in[k,a]$, then $k\leqslant\min(y,U(u,v))\leqslant a\leqslant\max(y,U(u,v))$, it follows that $U(u,v)=\max(u,v)$ and $F(y,U(u,v))=\min(y,U(u,v))$. Therefore, $\left(f\odot_{F}(g\odot_{U} h)\right)(z)=\left((f\odot_{F}g)\odot_{U}(f\odot_{F}h)\right)(z)$ from Theorem \ref{theor41} (ii).
  \end{enumerate}
\end{enumerate}

Case (d). If $z=F(y,U(u,v))\in[a,1)$, then $y\in[a,1]$ and $U(u,v)\in[a,1)$, we distinguish two subcases.
\begin{enumerate}
  \item []
  \begin{enumerate}
  \item [{\rm $\bullet$}] If $F(y,U(u,v))=a+(1-a)T\left(\frac{y-a}{1-a},\frac{U(u,v)-a}{1-a}\right)$ and $U=\max$, then $\left(f\odot_{F}(g\odot_{U} h)\right)(z)=\left((f\odot_{F}g)\odot_{U}(f\odot_{F}h)\right)(z)$ by Theorem \ref{theor41} (ii).
  \item [{\rm $\bullet$}] If $F(y,U(u,v))=a+(1-a)T\left(\frac{y-a}{1-a},\frac{U(u,v)-a}{1-a}\right)$ and $U(u,v)=a+(1-a)S\left(\frac{u-a}{1-a},\frac{v-a}{1-a}\right)$, where $T$ is a strict t-norm and $S$ is a nilpotent t-conorm. Due to $F(y,U(u,v))=U(F(y,u),F(y,v))$, from formulas (\ref{eqn411}) and (\ref{eqn412}), it holds that
\begin{eqnarray}\label{eqn413}
\left(f\odot_{F}(g\odot_{U} h)\right)(z)\leqslant\left((f\odot_{F}g)\odot_{U}(f\odot_{F}h)\right)(z).
\end{eqnarray}
Therefore, we just need to prove $\left(f\odot_{F}(g\odot_{U} h)\right)(z)\geqslant\left((f\odot_{F}g)\odot_{U}(f\odot_{F}h)\right)(z)$.
We first prove the following statement.

{\textbf A}. For any $p,q, s, t, y'\in [a,1]$, $U(F(p,q),F(s,t))=z$ and $U(F(y',q),F(y',t))=z$ imply $f(y')\wedge g(q)\wedge h(t)\geqslant f(p)\wedge g(q)\wedge f(s)\wedge h(t)$ whenever $U(q,t)<1$.

Since $U(F(p,q),F(s,t))=z$ and $U(F(y',q),F(y',t))=z$, we have $p\leqslant y'\leqslant s$ or $s\leqslant y'\leqslant p$. Otherwise, $y'<p\wedge s$ or $y'>p\vee s$. Say, $y'<p\wedge s$. Then
$z=U(F(p,q),F(s,t))\geqslant U(F(p\wedge s,q),F(p\wedge s,t))=F(p\wedge s,U(q,t))>F(y',U(q,t))=U(F(y',q),F(y',t))=z$ since $U(q,t)<1$ and $F$ is a strict t-norm on $[a,1]$, a contradiction. Consequently, $p\leqslant y'\leqslant s$ or $s\leqslant y'\leqslant p$. Therefore, $f(y')\geqslant f(p)\wedge f(s)$ since $f$ is convex, which means that $f(y')\wedge g(q)\wedge h(t)\geqslant f(p)\wedge g(q)\wedge f(s)\wedge h(t)$. This completes the proof of {\textbf A}.

Then, using {\textbf A}, we have that
\begin{eqnarray*}
\left((f\odot_{F} g)\odot_{U} (f\odot_{F} h)\right)(z)&=&\bigvee\limits_{U(F(p,q),F(s,t))=z}f(p)\wedge g(q)\wedge f(s)\wedge h(t)\\
&\leqslant&\bigvee\limits_{U(F(y',q),F(y',t))=z}f(y')\wedge g(q)\wedge h(t)\\
&=&\bigvee\limits_{F(y',U(q,t))=z}f(y')\wedge g(q)\wedge h(t)\\
&=&\bigvee\limits_{F(y,U(u,v))=z}f(y)\wedge g(u)\wedge h(v)\\
&=&\left(f\odot_{F}(g\odot_{U}h)\right)(z),
\end{eqnarray*}
i.e., $\left(f\odot_{F}(g\odot_{U} h)\right)(z)\geqslant\left((f\odot_{F} g)\odot_{U} (f\odot_{F} h)\right)(z)$.
 \end{enumerate}
\end{enumerate}

(iii) $e>k$ and $U,F$ are given as Eqs. (\ref{eqn49}) and (\ref{eqn410}), respectively. By Theorem \ref{theor41} (i) and (ii), it is obvious that $\left(f\odot_{F}(g\odot_{U} h)\right)(z)=\left((f\odot_{F}g)\odot_{U}(f\odot_{F}h)\right)(z)$ for $z\in[0,1)$.

Cases (i), (ii) and (iii) yield that
$$\left(f\odot_{F}(g\odot_{U} h)\right)(z)=\left((f\odot_{F} g)\odot_{U} (f\odot_{F} h)\right)(z) \mbox{ for any }z\in[0,1).$$

In the subsequent, we shall prove $\left(f\odot_{F}(g\odot_{U} h)\right)(1)=\left((f\odot_{F} g)\odot_{U} (f\odot_{F} h)\right)(1)$.

Indeed, if $F(y,U(u,v))=1$ and $U(F(p,q),F(s,t))=1$, then $y=1$ and $u\vee v=1$, $p=q=1$ or $s=t=1$. Thus from formulas (\ref{eqn411}) and (\ref{eqn412}),
\begin{eqnarray*}
\left(f\odot_{F}(g\odot_{U} h)\right)(1)&=&\left(\bigvee\limits_{u\in[0,1]}f(1)\wedge g(u)\wedge h(1)\right)\vee\left(\bigvee\limits_{v\in[0,1]}f(1)\wedge g(1)\wedge h(v)\right)
\end{eqnarray*}
and
\begin{eqnarray*}
\left((f\odot_{F}g)\odot_{U}(f\odot_{F}h)\right)(1)&=&\left(\bigvee\limits_{p,q\in[0,1]}f(p)\wedge g(q)\wedge f(1)\wedge h(1)\right)\\
&&\vee\left(\bigvee\limits_{s,t\in[0,1]}f(1)\wedge g(1)\wedge f(s)\wedge h(t)\right).
\end{eqnarray*}
Obviously, $\left(f\odot_{F}(g\odot_{U} h)\right)(1)\geqslant\left((f\odot_{F}g)\odot_{U}(f\odot_{F}h)\right)(1)$ since
$$\bigvee\limits_{u\in[0,1]}f(1)\wedge g(u)\wedge h(1)\geqslant\bigvee\limits_{p,q\in[0,1]}f(p)\wedge g(q)\wedge f(1)\wedge h(1)$$ and
$$\bigvee\limits_{v\in[0,1]}f(1)\wedge g(1)\wedge h(v)\geqslant\bigvee\limits_{s,t\in[0,1]}f(1)\wedge g(1)\wedge f(s)\wedge h(t).$$

On the other hand, from $y=1$ and $u\vee v=1$, $p=q=1$ or $s=t=1$, we also have $F(y,U(u,v))=1=U(F(y,u),F(y,v))$. Thus
\begin{eqnarray*}
\left(f\odot_{F}(g\odot_{U} h)\right)(1)&=&\bigvee\limits_{F(y,U(u,v))=1}f(y)\wedge g(u)\wedge h(v)\\
&=&\bigvee\limits_{U(F(y,u),F(y,v))=1}f(y)\wedge g(u)\wedge f(y)\wedge h(v) \\
&\leqslant &\bigvee\limits_{U(F(p,q),F(s,t))=1}f(p)\wedge g(q)\wedge f(s)\wedge h(t)\\
&=&\left((f\odot_{F}g)\odot_{U}(f\odot_{F}h)\right)(1).
\end{eqnarray*}
Consequently, $\left(f\odot_{F}(g\odot_{U} h)\right)(1)=\left((f\odot_{F} g)\odot_{U} (f\odot_{F} h)\right)(1)$.

In summary, $$(f\odot_{F}(g\odot_{U} h))(z)=((f\odot_{F}g)\odot_{U}(f\odot_{F}h))(z) \mbox{ for all }z\in[0,1].$$
\end{proof}

For a conjunctive uninorm $U\in\mathcal{WCU}$ with neutral element $e\in(0,1)$, we first have the following lemma.
\begin{lemma}[\cite{Dragan2015,LiG2015}]\label{theor414}
A continuous nullnorm $F$ with an absorbing element $k\in(0,1)$ and a conjunctive uninorm $U\in\mathcal{WCU}$ with neutral element $e\in(0,1)$ satisfy (CD) if and only if one of the following cases is fulfilled:
\begin{enumerate}
  \item []
  \begin{enumerate}
  \item [{\rm (i)}]
$e>k$ and $F,U$ are given as in \cite{Mas2002} (Proposition 4.3), i.e.,
\begin{eqnarray}\label{eqn415}
U(x,y)=\left\{\begin{array}{ll}
\max(x,y) , & {\mbox{\scriptsize\normalsize if }(x,y)\in[e,1]^{2},}\\
\min(x,y),    &  {\mbox{\scriptsize\normalsize otherwise }}
\end{array}
\right.
\end{eqnarray}
and
\begin{eqnarray}\label{eqn416}
F(x,y)=\left\{\begin{array}{ll}
kS_{1}(\frac{x}{k},\frac{y}{k}) , & {\mbox{\scriptsize\normalsize if }(x,y)\in[0,k]^{2},}\\
k+(e-k)T_{1}\left(\frac{x-k}{e-k},\frac{y-k}{e-k}\right),    & {\mbox{\scriptsize\normalsize if }(x,y)\in[k,e]^{2},}\\
e+(1-e)T_{2}\left(\frac{x-e}{1-e},\frac{y-e}{1-e}\right),    & {\mbox{\scriptsize\normalsize if }(x,y)\in[e,1]^{2},}\\
k,    & {\mbox{\scriptsize\normalsize if }\min(x,y)\leqslant k\leqslant \max(x,y),}\\
\min(x,y),    &  {\mbox{\scriptsize\normalsize otherwise, }}
\end{array}
\right.
\end{eqnarray}
 where $S_{1}$ is continuous t-conorm and $T_{1}$ and $T_{2}$ are continuous t-norms.
\end{enumerate}

\begin{enumerate}
  \item [{\rm (ii)}]
  $e>k$ and $F, U$ are given as in \cite{Dragan2013} (Theorem 17), i.e., there is $a\in[e,1)$ such that $F$ and $U$ are given by
  \begin{eqnarray}\label{eqn419}
U(x,y)=\left\{\begin{array}{ll}
\min(x,y) , & {\mbox{\scriptsize\normalsize if }(x,y)\in[0,e]\times[0,1]\cup[0,1]\times[0,e],}\\
a+(1-a)S\left(\frac{x-a}{1-a},\frac{y-a}{1-a}\right),    & {\mbox{\scriptsize\normalsize if }(x,y)\in[a,1]^{2},}\\
\max(x,y),    &  {\mbox{\scriptsize\normalsize otherwise}}
\end{array}
\right.
\end{eqnarray}
and
\begin{eqnarray}\label{eqn420}
F(x,y)=\left\{\begin{array}{ll}
kS_{1}(\frac{x}{k},\frac{y}{k}) , & {\mbox{\scriptsize\normalsize if }(x,y)\in[0,k]^{2},}\\
k+(e-k)T_{1}\left(\frac{x-k}{e-k},\frac{y-k}{e-k}\right),    & {\mbox{\scriptsize\normalsize if }(x,y)\in[k,e]^{2},}\\
e+(a-e)T_{2}\left(\frac{x-e}{a-e},\frac{y-e}{a-e}\right),    & {\mbox{\scriptsize\normalsize if }(x,y)\in[e,a]^{2},}\\
a+(1-a)T\left(\frac{x-a}{1-a},\frac{y-a}{1-a}\right),    & {\mbox{\scriptsize\normalsize if }(x,y)\in[a,1]^{2},}\\
k,    & {\mbox{\scriptsize\normalsize if }\min(x,y)\leqslant k\leqslant \max(x,y),}\\
\min(x,y),    &  {\mbox{\scriptsize\normalsize otherwise, }}
\end{array}
\right.
\end{eqnarray}
 where $S_{1}$ is continuous t-conorm and $T_{1}$ and $T_{1}$ are continuous t-norms. Moreover, $S$ is a nilpotent t-conorm such that the additive generator $s$ of $S$ satisfying $s(1)=1$ is also a multiplicative generator of the strict t-norm $T$.
  \end{enumerate}

   \item [{\rm (iii)}]
  $e<k$ and
  \begin{eqnarray}\label{eqn421}
U(x,y)=\left\{\begin{array}{ll}
\min(x,y) , & {\mbox{\scriptsize\normalsize if }(x,y)\in[0,e]^{2},}\\
\max(x,y),  & {\mbox{\scriptsize\normalsize if }(x,y)\in[e,1]^{2},}\\
1, & {\mbox{\scriptsize\normalsize if }x=1,y\neq0 \mbox{ or } x\neq0, y=1,}\\
\min(x,y),    &  {\mbox{\scriptsize\normalsize otherwise }}
\end{array}
\right.
\end{eqnarray}
and
  \begin{eqnarray}\label{eqn422}
F(x,y)=\left\{\begin{array}{ll}
eS_{1}\left(\frac{x}{e},\frac{y}{e}\right) , & {\mbox{\scriptsize\normalsize if }(x,y)\in[0,e]^{2},}\\
e+(k-e)S_{2}\left(\frac{x-e}{k-e},\frac{y-e}{k-e}\right),    & {\mbox{\scriptsize\normalsize if }(x,y)\in[e,k]^{2},}\\
k+(1-k)T\left(\frac{x-k}{1-k},\frac{y-k}{1-k}\right),    & {\mbox{\scriptsize\normalsize if }(x,y)\in[k,1]^{2},}\\
\max(x,y),    & {\mbox{\scriptsize\normalsize if }\min(x,y)\leqslant e\leqslant \max(x,y)\leqslant k,}\\
k,    &  {\mbox{\scriptsize\normalsize otherwise, }}
\end{array}
\right.
\end{eqnarray}
 where $S_{1}$ and $S_{2}$ are continuous t-conorms and $T$ is a continuous t-norm.
    \end{enumerate}
\end{lemma}

Then based on Lemma \ref{theor414}, by a completely similar proof to Theorem \ref{5theo3.4}, we have the following theorem.
\begin{theorem}
Let $F$ be a continuous nullnorm with an absorbing element $k\in(0,1)$ and $U\in\mathcal{WCU}$ a conjunctive uninorm with neutral element $e\in(0,1)$ satisfying (CD). If $f\in \mathcal{F}$ is convex, then the following holds for all $g,h\in\mathcal{F}$.
$$f\odot_{F}(g\odot_{U} h)=(f\odot_{F}g)\odot_{U}(f\odot_{F}h).$$
\end{theorem}

\section{Distributivity of extended continuous operators and uninorms}

A form of the relaxed nullnorm that is obtained by omitting commutativity and associativity from Definition \ref{defi2.4} was introduced in \cite{Drewniak2008}. The set of all such type of operators is denoted by $Z_{k}$ where $k$ is an absorbing element of such opertors.
In order to investigate the distributive laws between the extended continuous operators and uninorms, we first need the following three lemmas.
\begin{lemma}[\cite{Dragan2013}]\label{lem43}
A continuous operator $F\in Z_{k}$ and a continuous t-conorm  $S$ satisfy (CDl) if and only if exactly one of the following cases is fulfilled:

(i) $S=S_{M}$,

(ii) there is an $a\in[k,1)$ such that $S,F$ are given by
\begin{eqnarray}\label{eqnarr42}
S(x,y)=\left\{\begin{array}{ll}
a+(1-a)S_{L}\left(\frac{x-a}{1-a},\frac{y-a}{1-a}\right),    & {\mbox{\scriptsize\normalsize if }(x,y)\in[a,1]^{2},}\\
\max(x,y),    &  {\mbox{\scriptsize\normalsize otherwise}}
\end{array}
\right.
\end{eqnarray}
and
\begin{eqnarray}\label{eqnarr402}
F=\left\{\begin{array}{ll}
A,    & {\mbox{\scriptsize\normalsize on }[0,k]^{2},}\\
B,    & {\mbox{\scriptsize\normalsize on }[k,1]^{2},}\\
k,    &  {\mbox{\scriptsize\normalsize otherwise, }}
\end{array}
\right.
\end{eqnarray}
where $A:[0,k]^{2}\rightarrow[0,k]$ is a continuous increasing operator with neutral element $0$, $B:[k,1]^{2}\rightarrow[k,1]$ is a continuous increasing operator with neutral element $1$ such that $B(x,y)\in[a,1]$ for all $x,y\in[a,1]$ and $B=T_{P}$ on $[a,1]^{2}$.
\end{lemma}

\begin{lemma}[\cite{Dragan2013}]\label{lem44}
A continuous operator $F\in Z_{k}$ and a uninorm $U\in U_{\max}\cap\mathcal{WCU}$ with neutral element $e$ satisfy (CDl) if and only if $e<k$ and exactly one of the following cases is fulfilled:

(i) $F$ and $U$ are given as in \cite{Drewniak2008} (Theorem 16), i.e., $U=\overline{U}$ and
\begin{eqnarray}\label{eqn427}
F=\left\{\begin{array}{ll}
A_{1},    & {\mbox{\scriptsize\normalsize on }[0,e]^{2},}\\
A_{2},    & {\mbox{\scriptsize\normalsize on }[e,k]^{2},}\\
A_{3},    & {\mbox{\scriptsize\normalsize on }[0,e]\times[e,k],}\\
\max,    & {\mbox{\scriptsize\normalsize on }[e,k]\times[0,e],}\\
B,    & {\mbox{\scriptsize\normalsize on }[k,1]^{2},}\\
k,    &  {\mbox{\scriptsize\normalsize otherwise, }}
\end{array}
\right.
\end{eqnarray}
where $0$ is neutral element of the operator $A_{1}$ and a left side neutral element of $A_{3}$, $1$ is a neutral element of $B$, $e$ is a right side neutral element of $A_{2}$ and where $A_{1}$, $A_{2}$, $A_{3}$, $B$ are continuous increasing operators.

(ii)  there is an $a\in[k,1)$ such that $U$ is given by (\ref{eqn47})and $F$ is given by (\ref{eqn427}) such that $B(x,y)\in[a,1]$ for all $x,y\in[a,1]$ and $B=T_{P}$ on $[a,1]^{2}$.
\end{lemma}
\begin{lemma}[\cite{Dragan2013}]\label{lem45}
A continuous operator $F\in Z_{k}$ and a uninorm $U\in U_{\min}\cap\mathcal{WCU}$ with neutral element $e\in(0,1)$ satisfy (CDl) if and only if $k<e$ and exactly one of the following cases is fulfilled:

(i) $F$ and $U$ are given as in \cite{Drewniak2008} (Theorem 18), i.e., $U=\underline{U}$ and
\begin{eqnarray}\label{eqn428}
F=\left\{\begin{array}{ll}
A,    & {\mbox{\scriptsize\normalsize on }(x,y)\in[0,k]^{2},}\\
B_{1},    & {\mbox{\scriptsize\normalsize on }(x,y)\in[k,e]^{2},}\\
B_{3},    & {\mbox{\scriptsize\normalsize on }(x,y)\in[e,1]\times[k,e],}\\
\min,    & {\mbox{\scriptsize\normalsize on }(x,y)\in[k,e]\times[e,1],}\\
B_{2},    & {\mbox{\scriptsize\normalsize on }(x,y)\in[e,1]^{2},}\\
k,    &  {\mbox{\scriptsize\normalsize otherwise, }}
\end{array}
\right.
\end{eqnarray}
where $0$ is neutral element of $A$, $1$ is neutral element of $B$, and a left side neutral element of $B_{3}$, $e$ is a right side neutral element of $B_{1}$, and where $B_{1}$, $B_{2}$, $B_{3}$, $A$ are continuous increasing operators.

(ii)  there is an $a\in[e,1)$ such that $U$ is given by (\ref{eqn419})and $F$ is given by (\ref{eqn428}) such that $B_{2}(x,y)\in[a,1]$ for all $x,y\in[a,1]$ and $B=T_{P}$ on $[a,1]^{2}$.
\end{lemma}

Consider the distributivity between $\odot_{F}$ and $\odot_{S}$ (where $\odot_{S}$ is an extended t-conorm) on fuzzy truth values. Then we have the following result.

\begin{theorem}\label{the41}
 Let $F\in Z_{k}$ be a continuous operator and $S$ a continuous t-conorm satisfying (CDl). If $f\in\mathcal{F}$ is convex, then the following holds for all $g,h\in\mathcal{F}$.
$$f\odot_{F}(g\odot_{S} h)=(f\odot_{F}g)\odot_{S}(f\odot_{F}h).$$
\end{theorem}
\begin{proof}First, from formulas (\ref{eqn41}) and (\ref{eqn42}), we have that
\begin{eqnarray}\label{eqn401}
\left(f\odot_{F}(g\odot_{S} h)\right)(z)=\bigvee\limits_{F(y,S(u,v))=z}f(y)\wedge g(u)\wedge h(v)
\end{eqnarray}
and
\begin{eqnarray}\label{eqn402}
\left((f\odot_{F}g)\odot_{S}(f\odot_{F}h)\right)(z)&=&\bigvee\limits_{S(F(p,q),F(s,t))=z}f(p)\wedge g(q)\wedge f(s)\wedge h(t).
\end{eqnarray}

Now, suppose that $z\in [0,1)$. Then from $z=F(y,S(u,v))\in [0,1)$, we have $S(u,v)<1$ for any $u,v\in [0,1]$. Thus $F(y,S(u,v))=S(F(y,u),F(y,v))\in [0,1)$ since $F$ and $S$ satisfy (CDl).

Next, we divide our proof into two cases as follows from Lemma \ref{lem43}.

(i) If $S=S_{M}$, then $\left(f\odot_{F}(g\odot_{S} h)\right)(z)=\left((f\odot_{F} g)\odot_{S} (f\odot_{F} h)\right)(z)$ for $z\in [0,1)$ by Theorem \ref{theor41} (ii).

(ii) If there is an $a\in [k,1)$ such that $S$ and $F$ are given by formulas (\ref{eqnarr42}) and (\ref{eqnarr402}), respectively, then we distinguish three subcases.

Case (a). If $z=F(y,S(u,v))\in[0,k]$, then $F=A$ and $S=S_{M}$. Consequently, $\left(f\odot_{F}(g\odot_{S} h)\right)(z)=\left((f\odot_{F} g)\odot_{S} (f\odot_{F} h)\right)(z)$ by Theorem \ref{theor41} (ii).

Case (b). If $z=F(y,S(u,v))\in[k,1)$, then we have that either $z=F(y,S(u,v))\in[k,a]$ or $z=F(y,S(u,v))\in[a,1)$.
\begin{enumerate}
  \item []
  \begin{enumerate}
  \item [{\rm $\bullet$}] If $z=F(y,S(u,v))\in[k,a]$, then $F=B$ and $S=S_{M}$. Consequently, $\left(f\odot_{F}(g\odot_{S} h)\right)(z)=\left((f\odot_{F} g)\odot_{S} (f\odot_{F} h)\right)(z)$ by Theorem \ref{theor41} (ii).
  \item [{\rm $\bullet$}] If $z=F(y,S(u,v))\in[a,1)$, then $F=B=T_{P}$ and $S(u,v)=a+(1-a)S_{L}\left(\frac{u-a}{1-a},\frac{v-a}{1-a}\right)$ with $u,v\in[a,1]$. Due to $F(y,S(u,v))=S(F(y,u),F(y,v))$, from formulas (\ref{eqn401}) and (\ref{eqn402}), it holds that
\begin{eqnarray}\label{eqn403}
\left(f\odot_{F}(g\odot_{S} h)\right)(z)\leqslant\left((f\odot_{F}g)\odot_{S}(f\odot_{F}h)\right)(z).
\end{eqnarray}
Therefore, we just need to prove $\left(f\odot_{F}(g\odot_{S} h)\right)(z)\geqslant\left((f\odot_{F}g)\odot_{S}(f\odot_{F}h)\right)(z)$. We first prove the following statement whose proof is completely similar to {\textbf A}.

{\textbf B}. For any $p,q, s, t, y'\in [a,1)$, $S(F(p,q),F(s,t))=z$ and $S(F(y',q),F(y',t))=z$ imply $f(y')\wedge g(q)\wedge h(t)\geqslant f(p)\wedge g(q)\wedge f(s)\wedge h(t)$ whenever $S(q,t)<1$.

Then, using {\textbf B}, we have $\left(f\odot_{F}(g\odot_{S} h)\right)(z)\geqslant\left((f\odot_{F} g)\odot_{S} (f\odot_{F} h)\right)(z)$, which together with formula (\ref{eqn403}) yields that $$\left(f\odot_{F}(g\odot_{S} h)\right)(z)=\left((f\odot_{F} g)\odot_{S} (f\odot_{F} h)\right)(z)\mbox{ for }z\in [a,1).$$
  \end{enumerate}
\end{enumerate}

Case (c). If $z=F(y,S(u,v))=k$, then we have either $y=k$ and $S(u,v)\in[0,1)$ or $y\in[0,1]$ and $S(u,v)=k$ since $k$ is an absorbing element of $F$. \begin{enumerate}
  \item []
  \begin{enumerate}
  \item [{\rm$\bullet$}] If $y=k$ and $S(u,v)\in[0,1)$, then $S=S_{M}$ since $S(F(y,u),F(y,v))=F(y,S(u,v))=k$ and $k\leqslant a$. Therefore, $\left(f\odot_{F}(g\odot_{S} h)\right)(z)=\left((f\odot_{F} g)\odot_{S} (f\odot_{F} h)\right)(z)$ by Theorem \ref{theor41} (ii).
  \item [{\rm$\bullet$}] If $y\in[0,1]$ and $S(u,v)=k$, then $S=S_{M}$ since $S(u,v)=k$ and $k\leqslant a$. Therefore, $\left(f\odot_{F}(g\odot_{S} h)\right)(z)=\left((f\odot_{F} g)\odot_{S} (f\odot_{F} h)\right)(z)$ by Theorem \ref{theor41} (ii).
  \end{enumerate}
\end{enumerate}

Cases (a), (b) and (c) deduce that $\left(f\odot_{F}(g\odot_{S} h)\right)(z)=\left((f\odot_{F} g)\odot_{S} (f\odot_{F} h)\right)(z)$ for any $z\in[0,1)$.

Therefore, with Cases (i) and (ii), we have that
$$\left(f\odot_{F}(g\odot_{S} h)\right)(z)=\left((f\odot_{F} g)\odot_{S} (f\odot_{F} h)\right)(z) \mbox{ for any } z\in[0,1).$$

If $z=1$, then, in a similar way to the proof of Theorem \ref{5theo3.4}, we can get $\left(f\odot_{F}(g\odot_{S} h)\right)(1)=\left((f\odot_{F} g)\odot_{S} (f\odot_{F} h)\right)(1)$.

In summary, $$\left(f\odot_{F}(g\odot_{S} h)\right)(z)=\left((f\odot_{F} g)\odot_{S} (f\odot_{F} h)\right)(z) \mbox{ for any } z\in[0,1].$$
\end{proof}

Moreover, by using Lemmas \ref{lem44} and \ref{lem45}, respectively, the following two distributive laws can be derived in complete analogy to the proof of Theorem \ref{the41}.
\begin{theorem}\label{th3.7}
Let $F\in Z_{k}$ be a continuous operator and $U\in U_{\max}\cap\mathcal{WCU}$ a uninorm with neutral element $e$ satisfying (CDl). If $f\in\mathcal{F}$ is convex, then the following holds for all $g,h\in\mathcal{F}$.
$$f\odot_{F}(g\odot_{U} h)=(f\odot_{F}g)\odot_{U}(f\odot_{F}h).$$
\end{theorem}

\begin{theorem}\label{th3.8}
Let $F\in Z_{k}$ be a continuous operator and $U\in U_{\min}\cap\mathcal{WCU}$ a uninorm with neutral element $e$ satisfying (CDl). If $f\in\mathcal{F}$ is convex, then the following holds for all $g,h\in\mathcal{F}$.
$$f\odot_{F}(g\odot_{U} h)=(f\odot_{F}g)\odot_{U}(f\odot_{F}h).$$
\end{theorem}

\begin{remark}\cite{Dragan2013}\label{remar3.2}
Three lemmas for (CDr) are analogous to the presented ones for (CDl) and, therefore, omitted.
\end{remark}

Furthermore, from Remark \ref{remar3.2}, and similar to Theorems \ref{the41}, \ref{th3.7} and \ref{th3.8}, the following three distributive laws are hold.
\begin{theorem}
Let $F\in Z_{k}$ be a continuous operator and $S$ a continuous t-conorm satisfying (CDr). If $h\in\mathcal{F}$ is convex, then the following holds for all $f,g\in\mathcal{F}$.
$$(f\odot_{F}g)\odot_{S} h=(f\odot_{F}h)\odot_{S}(g\odot_{F}h).$$
\end{theorem}

\begin{theorem}
Let $F\in Z_{k}$ be a continuous operator and $U\in U_{\max}\cap\mathcal{WCU}$ a uninorm with neutral element $e$ satisfying (CDr). If $h\in\mathcal{F}$ is convex, then the following holds for all $f,g\in\mathcal{F}$.
$$(f\odot_{F}g)\odot_{U} h=(f\odot_{F}h)\odot_{U}(g\odot_{F}h).$$
\end{theorem}

\begin{theorem}
Let $F\in Z_{k}$ be a continuous operator and $U\in U_{\min}\cap\mathcal{WCU}$ a uninorm with neutral element $e$ satisfying (CDr). If $h\in\mathcal{F}$ is convex, then the following holds for all $f,g\in\mathcal{F}$.
$$(f\odot_{F}g)\odot_{U} h=(f\odot_{F}h)\odot_{U}(g\odot_{F}h).$$
\end{theorem}

\section{Conclusions}

The main contributions are the distributive laws between the extended nullnorm and uninorm on fuzzy truth values under the condition that the nullnorm is conditionally distributive over the uninorm, and the left and right distributive laws between the extended generalization nullnorms and uninorms. The results in this paper generalize the corresponding ones in \cite{Harding2016,Xie2018,Walker2005}.

\end{document}